\documentclass[mathpazo]{paper}

\usepackage[T1]{fontenc}
\usepackage[utf8]{inputenc}
\usepackage{amsmath,amssymb,mathtools}
\usepackage{graphicx}
\usepackage{booktabs,multirow}
\usepackage{microtype}
\usepackage{enumitem}
\usepackage{float}
\allowdisplaybreaks

\newcommand{\R}{\mathbb{R}}

\newcommand{\Nzero}{\mathbb{N}_0}
\newcommand{\Npos}{\mathbb{N}_+}
\newcommand{\Z}{\mathbb{Z}}
\newcommand{\ip}[2]{\left\langle #1,#2\right\rangle}
\newcommand{\sn}[1]{\left|#1\right|_L}
\newcommand{\calR}{\mathcal{R}}
\newcommand{\calP}{\mathcal{P}}
\newcommand{\calQ}{\mathcal{Q}}
\newcommand{\Up}{\Upsilon}
\newcommand{\dd}{\mathrm{d}}
\newcommand{\diag}{\operatorname{diag}}
\newcommand{\cC}{\mathcal{C}}
\newcommand{\hR}{\widehat{\mathbb R}}

\makeatletter
\let\c@lemma\c@theorem

\let\p@lemma\p@theorem
\let\c@proposition\c@theorem

\let\p@proposition\p@theorem
\let\c@corollary\c@theorem

\let\p@corollary\p@theorem
\let\c@example\c@theorem

\let\p@example\p@theorem
\let\c@remark\c@theorem

\let\p@remark\p@theorem
\makeatother

\title[Energy Laws for First-Subdiagonal Pad\'e Methods]{On Energy Laws and Stability of\\First-Subdiagonal Pad\'e Approximants for\\Linear Seminegative Problems}

\author[]{Miaosen Jiao\affil{1},
Zheng Sun\affil{2}~and Kailiang Wu\affil{3}\comma\corrauth}
\address{\affilnum{1}\ Department of Mathematics, Southern University of Science and Technology, Shenzhen 518055, China \\
\affilnum{2}\ Department of Mathematics, The University of Alabama, Tuscaloosa, AL 35487, USA \\
\affilnum{3}\ Department of Mathematics and Shenzhen International Center for Mathematics, Southern University of Science and Technology, Shenzhen 518055, China}
\emails{{\tt 12531031@mail.sustech.edu.cn} (M.~Jiao), {\tt zsun30@ua.edu} (Z.~Sun), {\tt wukl@sustech.edu.cn} (K.~Wu)}

\begin{document}

\begin{abstract}
We derive an explicit discrete energy identity for rational time discretizations generated by the first-subdiagonal Pad\'e approximants of the exponential for solving linear seminegative problems.  This work extends the diagonal Pad\'e energy laws in [Z. Sun, Y. Wei, and K. Wu, SIAM J. Numer. Anal., 60 (2022)] to the first-subdiagonal family.  The main new ingredient is an explicit Cholesky-type factorization of the energy coefficient matrix associated with the semi-inner-product terms in the discrete energy identity.  The construction and proof of this factorization are nontrivial, since the matrix entries are alternating sums of Pad\'e coefficients and the triangular factor has a parity-dependent factorial structure.  We prove the factorization by reducing it to scalar rational identities and establishing them through finite product reductions and telescoping summations.  Together with a \(\beta\)-coefficient cancellation, the factorization yields an exact discrete energy law that recovers the classical unconditional contractivity for linear seminegative problems.  Numerical experiments adapted from the diagonal Pad\'e energy-law setting illustrate the predicted order and verify the discrete dissipation identity.
\end{abstract}


\maketitle

\section{Introduction}
\label{sec:introduction}

Spatial semidiscretization of dissipative or conservative linear time-dependent partial differential equations (PDEs) often yields autonomous linear systems whose discrete operators are seminegative with respect to a specific inner product. Such systems satisfy a continuous energy dissipation law: the squared norm induced by the chosen inner product is nonincreasing in time. For high-order time discretizations, it is therefore natural to ask not only whether the method is accurate, but also whether the discrete scheme preserves---or at least precisely quantifies---the continuous dissipation mechanism, which can be essential to the robustness of the methods. In this paper, we focus particularly on the first-subdiagonal Pad\'e approximants.

Classical stability theory for Runge--Kutta methods often begins with the Dahlquist scalar test equation and its associated stability function. This approach leads to stability-region criteria such as A-stability and remains a central tool in the analysis of stiff time integrators~\cite{HairerWanner1996SolvingODEII}. For semidiscrete PDE systems, however, stability is also tied to the energy induced by the underlying inner product. Energy methods therefore provide a complementary viewpoint: they work directly with the discrete change of energy and identify the dissipative terms generated by the time discretization. Early energy-based stability analyses include the theory of algebraically stable Runge--Kutta methods~\cite{HairerWanner1996SolvingODEII} and studies of explicit Runge--Kutta methods for coercive problems~\cite{spijker1983contractivity,LevyTadmor1998EnergyMethod,gottlieb2001strong} (a subclass of seminegative problems). Energy arguments have also been widely used in the analysis of fully discrete schemes for linear time-dependent PDEs in various settings. 

A substantial literature on explicit Runge--Kutta methods for general linear seminegative problems has developed in recent decades. In this setting, classical scalar stability analysis may provide necessary but not sufficient conditions for strong stability in the $L^2$ norm~\cite{iserles2009first,SunShu2017RK4}; hence, the energy method provides a particularly useful alternative approach. In~\cite{Tadmor2002EnergyMethodII}, Tadmor proved the strong stability of the third-order explicit Runge--Kutta method; in~\cite{SunShu2017RK4}, Sun and Shu disproved the strong stability of the classic fourth-order Runge--Kutta method, but proved the method is strongly stable over every pair of consecutive time steps; in~\cite{RanochaOffner2018L2Stability}, Ranocha and \"Offner proved the strong stability of a ten-stage fourth-order Runge--Kutta method. Based on these works, Sun and Shu~\cite{SunShu2019StrongStability} provided a unified framework for analyzing the strong stability of explicit Runge--Kutta methods of arbitrarily high order. Moreover, the work in~\cite{SunShu2019StrongStability} provided a systematic algebraic framework for deriving discrete energy laws for Runge--Kutta methods. See also~\cite{XuZhangShuWang2019RKDGStability}. Subsequent work extended this line of analysis, including studies on nonlinear problems~\cite{Ranocha2021StrongStabilityNonlinear} and connections with hypocoercivity and semidissipative matrix structures~\cite{achleitner2022necessary,AchleitnerArnoldJuengel2023Hypocoercivity}. In a different but related direction, relaxation Runge--Kutta methods and their extensions provide mechanisms for enforcing conservation, energy stability, or entropy stability in nonlinear and fully discrete settings~\cite{Ketcheson2019RelaxationRK,RanochaKetcheson2020EnergyStability,Ranocha2021StrongStabilityNonlinear,RanochaSayyariDalcinParsaniKetcheson2020RelaxationEntropy}. Together, these works show that energy-based analysis is useful not only for proving stability, but also for identifying structural mechanisms behind stability properties of Runge--Kutta discretizations.

The closest point of departure for the present paper is the work of Sun, Wei, and Wu~\cite{SunWeiWu2022EnergyLaws}, which extended the Runge--Kutta energy framework in~\cite{SunShu2019StrongStability} to general implicit and explicit methods. In particular, they established a unified discrete energy law for all diagonal Pad\'e approximants. A key feature of their analysis is that the exact dissipative terms are obtained from an explicit Cholesky-type decomposition of a highly structured symmetric coefficient matrix. Thus, the problem is not only to prove contractivity, but to uncover the algebraic structure responsible for the discrete energy dissipation.

Pad\'e approximants of the exponential are natural objects in this context. They arise as rational stability functions of important implicit Runge--Kutta methods~\cite{Ehle1973AStablePade,Butcher1977AStableImplicitRK,WannerHairerNorsett1978OrderStars} and are also fundamental in algorithms for the matrix exponential, including scaling-and-squaring methods and subdiagonal Pad\'e variants~\cite{Higham2005ScalingSquaring,GuttelNakatsukasa2016SubdiagonalPade}. More recently, subdiagonal Pad\'e approximants have been studied in the rational approximation of operator semigroups~\cite{EgertRozendaal2013SubdiagonalSemigroups,NeubranderOzerWindsperger2020SubdiagonalPade}. For the first-subdiagonal family, classical A-stability and zero-location results are known~\cite{Ehle1973AStablePade,EhlePicel1975TwoParameter,SaffVarga1975}.

The $s$-stage Radau IIA collocation method is a classical algebraically stable implicit Runge--Kutta method whose stability function is exactly the $(s-1,s)$ Pad\'e approximant to the exponential~\cite{HairerWanner1996SolvingODEII}. Its algebraic stability implies $B$-stability and hence stepwise contractivity between numerical solutions of dissipative ODEs; the standard argument is formulated in terms of the internal Runge--Kutta stages~\cite{HairerWanner1981AlgebraicallyStable,HairerWanner1996SolvingODEII}. In related RK--SAV work, Radau IIA formulas were used in extrapolated schemes for the Allen--Cahn and Cahn--Hilliard equations, and a discrete modified-energy decay law was established for the resulting schemes~\cite{AkrivisLiLi2019EnergyDecaying}. Taken together, these studies establish contractivity or modified-energy decay through stage-based or auxiliary-variable formulations. They do not, however, yield a unified closed-form factorization of the Pad\'e energy coefficient matrix valid for the entire first-subdiagonal family.

The present paper addresses this gap by extending the diagonal Pad\'e energy-law analysis to the first-subdiagonal Pad\'e family. This extension is not a direct consequence of the diagonal case. The energy coefficient matrix for the first-subdiagonal family is built from alternating sums of Pad\'e coefficients, while the triangular factor identified below has parity-dependent factorial entries. Establishing the factorization therefore requires new finite product identities, adjacent-index summations, and a rational-continuation argument. To the best of our knowledge, an explicit Cholesky-type energy-law structure for the first-subdiagonal family has not previously been available.

The main contributions are as follows.
\begin{itemize}[leftmargin=2em]
	\item We identify and prove an explicit Cholesky-type factorization of the energy coefficient matrix for the first-subdiagonal Pad\'e family. This factorization is the central algebraic structure of the paper. Its proof is nontrivial because the matrix entries are alternating sums of Pad\'e coefficients, whereas the triangular factor has parity-dependent factorial entries.
	\item We combine this factorization with a \(\beta\)-coefficient cancellation and the Pad\'e denominator root-location property to obtain an exact discrete energy law. This identity recovers the classical unconditional contractivity and, more importantly, explicitly identifies the individual dissipation channels generated by the time discretization.
	\item We develop a rational-continuation and finite-telescoping proof of the factorization. The argument reduces the matrix identity to scalar rational identities and proves them through product reductions and adjacent-index summations.
\end{itemize}

Numerical experiments, adapted from the diagonal Pad\'e energy-law setting, illustrate the expected convergence order and verify the stepwise discrete dissipation identity.

The paper is organized as follows. Section~\ref{sec:general-framework} states the general energy framework and denominator-invertibility criterion. Section~\ref{sec:first-subdiagonal} introduces the first-subdiagonal Pad\'e coefficients, proves denominator well-definedness, states the Cholesky-type factorization and the \(\beta\)-coefficient cancellation, and derives the resulting energy law. Section~\ref{sec:first-subdiagonal-proof} proves the algebraic factorization and the auxiliary identities used in Section~\ref{sec:first-subdiagonal}. Section~\ref{sec:numerical-experiments} presents three numerical tests, adapted from the examples in~\cite{SunWeiWu2022EnergyLaws}, for the first-subdiagonal Pad\'e family. Section~\ref{sec:conclusions} summarizes the results and points to possible extensions.

\section{General energy framework}
\label{sec:general-framework}
We first recall the components of the Runge--Kutta energy framework for linear seminegative problems that are needed below.  The identities and criteria in this section are adapted from~\cite{SunWeiWu2022EnergyLaws} and are stated in the notation needed for the first-subdiagonal Pad\'e analysis below.
We use $\Nzero=\{0,1,2,\ldots\}$ and $\Npos=\{1,2,\ldots\}$. All vector spaces are real unless otherwise specified.

\subsection{Linear seminegative systems}
Let $V$ be a real Hilbert space equipped with the inner product
$\ip{\cdot}{\cdot}$ and the induced norm $\|\cdot\|$. Let $\mathcal L(V)$ denote the space of bounded linear operators on $V$, with $\|A\|_{\mathcal L(V)}:=\sup_{0\ne v\in V}\|Av\|/\|v\|$. We consider the linear autonomous system
\begin{equation}
    \frac{\dd}{\dd t}u(t)=Lu(t),\qquad u(t)\in V,
    \label{eq:linear-system}
\end{equation}
where $L:V\to V$ is a bounded linear seminegative operator, namely
\begin{equation}
    \ip{Lv}{v}\le 0,\qquad \forall v\in V.
\end{equation}
The operator $L$ is not assumed to be normal. Following~\cite{SunWeiWu2022EnergyLaws}, we define the $L$-associated semi-inner product by
\begin{equation}
    [w,v]_L := -\ip{Lw}{v}-\ip{w}{Lv},
    \label{eq:semi-inner-product}
\end{equation}
and denote the corresponding seminorm by
\begin{equation}
    \sn{v}:=[v,v]_L^{1/2}.
\end{equation}
Then $[\,\cdot,\cdot\,]_L$ is a symmetric positive semidefinite bilinear form and $\sn{v}^2\ge 0$ for all $v\in V$. The exact solution of \eqref{eq:linear-system} satisfies the energy dissipation law
\begin{equation}
    \frac{\dd}{\dd t}\|u(t)\|^2
    =\ip{Lu(t)}{u(t)}+\ip{u(t)}{Lu(t)}
    =-\sn{u(t)}^2\le 0.
    \label{eq:continuous-energy}
\end{equation}

\subsection{A general energy identity for RK methods with rational stability functions}
Consider a one-step discretization whose action on the linear autonomous problem is described by a rational stability function
\begin{equation}
    \calR(z)=\frac{\calP(z)}{\calQ(z)},
    \qquad
    \calP(z)=\sum_{i=0}^{s}\theta_i z^i,
    \qquad
    \calQ(z)=\sum_{i=0}^{s}\vartheta_i z^i.
    \label{eq:rational-stability}
\end{equation}
Here $s$ is chosen large enough to cover both polynomial degrees, and coefficients beyond the actual degree are understood to be zero.
For RK methods, this is the standard stability-function representation. The fully discrete method applied to \eqref{eq:linear-system} can be written as
\begin{equation}
    u^{n+1}=\calR(\tau L)u^n
    =\calQ(\tau L)^{-1}\calP(\tau L)u^n,
    \label{eq:rk-discrete}
\end{equation}
where $\tau>0$ is the time step. Assuming $\calQ(\tau L)$ is invertible, set
\begin{equation}
    P:=\calP(\tau L),\qquad Q:=\calQ(\tau L),\qquad w^n:=Q^{-1}u^n.
\end{equation}
Since $P$ and $Q$ are polynomials in $L$, $P$ commutes with $Q$ and hence with $Q^{-1}$. Thus $u^{n+1}=Pw^n$ and $u^n=Qw^n$. Define
\begin{equation}
    \alpha_{ij}:=\theta_i\theta_j-\vartheta_i\vartheta_j,
    \qquad 0\le i,j\le s .
    \label{eq:alpha-def}
\end{equation}

Equivalently, the coefficients \(\alpha_{ij}\) are characterized by the two-variable identity
\begin{equation}
    \calP(x)\calP(y)-\calQ(x)\calQ(y)
    =
    \sum_{i=0}^{s}\sum_{j=0}^{s}\alpha_{ij}x^i y^j .
    \label{eq:alpha-bilinear-generating}
\end{equation}
In particular, setting \(y=-x\) gives
\begin{equation}
    \calP(x)\calP(-x)-\calQ(x)\calQ(-x)
    =
    \sum_{i=0}^{s}\sum_{j=0}^{s}(-1)^j\alpha_{ij}x^{i+j}.
    \label{eq:pq-minus-specialization}
\end{equation}
This specialization will be used below to identify the norm coefficients in the energy identity.

The next lemma is the algebraic Runge--Kutta energy identity from~\cite{SunWeiWu2022EnergyLaws}, written here for the rational stability function $\calR=\calP/\calQ$. It is included to fix notation; the Pad\'e-specific coefficient-matrix decomposition is constructed in Section~\ref{sec:first-subdiagonal}.
\begin{lemma}[General RK energy identity \cite{SunWeiWu2022EnergyLaws}]\label{lem:quoted-general-energy}
The numerical solution of \eqref{eq:rk-discrete} satisfies
\begin{equation}
\begin{aligned}
    \|u^{n+1}\|^2-\|u^n\|^2
    &=\sum_{k=0}^{s}\beta_k\tau^{2k}\|L^k w^n\|^2 \\
    &\quad +\sum_{i=0}^{s-1}\sum_{j=0}^{s-1}
    \gamma_{ij}\tau^{i+j+1}[L^i w^n,L^j w^n]_L,
\end{aligned}
\label{eq:discrete-energy-general}
\end{equation}
where
\begin{equation}
    \beta_k
    =
    \sum_{\ell=\max\{0,2k-s\}}^{\min\{2k,s\}}
    (-1)^{k-\ell}\alpha_{\ell,2k-\ell},
    \label{eq:beta-def}
\end{equation}
and
\begin{equation}
    \gamma_{ij}
    =
    \sum_{\ell=\max\{0,i+j+1-s\}}^{\min\{i,j\}}
    (-1)^{\min\{i,j\}+1-\ell}
    \alpha_{\ell,i+j+1-\ell}.
    \label{eq:gamma-def}
\end{equation}
\end{lemma}

For completeness, we recall the idea behind Lemma~\ref{lem:quoted-general-energy}. One first expands
\begin{equation}
\begin{aligned}
    \|u^{n+1}\|^2-\|u^n\|^2
    &=\|Pw^n\|^2-\|Qw^n\|^2  \\
    &=\sum_{i=0}^{s}\sum_{j=0}^{s}
    \alpha_{ij}\tau^{i+j}\ip{L^i w^n}{L^j w^n}.
\end{aligned}
\label{eq:raw-energy}
\end{equation}
The signs of the inner products in \eqref{eq:raw-energy} are generally indefinite. Repeated use of the discrete integration-by-parts identity
\begin{equation}
    \ip{w}{Lv}=-\ip{Lw}{v}-[w,v]_L
    \label{eq:discrete-ibp}
\end{equation}
transforms \eqref{eq:raw-energy} into \eqref{eq:discrete-energy-general}; comparison of coefficients gives \eqref{eq:beta-def}--\eqref{eq:gamma-def}.

Let
\begin{equation}
    B:=\diag(\beta_0,\ldots,\beta_s),
    \qquad
    \Up:=(\gamma_{ij})_{i,j=0}^{s-1}.
\end{equation}
The following decomposition principle is the Cholesky-type algebraic step used in the same energy framework~\cite{SunWeiWu2022EnergyLaws}. It converts a representation of the form $\Up=-U^TDU$ into an explicit energy-dissipation identity.
\begin{lemma}[Decomposition principle]\label{lem:quoted-factorization}
Assume that
\begin{equation}
    \Up=-U^TDU,
    \label{eq:general-cholesky}
\end{equation}
where $D=\diag(d_0,\ldots,d_{s-1})$ with $d_k\ge0$ and $U=(\mu_{ij})_{i,j=0}^{s-1}$ is upper triangular. Then
\begin{equation}
\sum_{i=0}^{s-1}\sum_{j=0}^{s-1}
\gamma_{ij}\tau^{i+j+1}[L^i w^n,L^j w^n]_L
=
-\sum_{k=0}^{s-1}d_k\tau^{2k+1}\sn{L^k u^{(k)}}^2,
\label{eq:factorized-energy-term}
\end{equation}
where
\begin{equation}
    u^{(k)}:=\sum_{j=k}^{s-1}\mu_{kj}(\tau L)^{j-k}w^n.
\end{equation}
\end{lemma}

\begin{proof}
Substitute \eqref{eq:general-cholesky} into the quadratic form in \eqref{eq:discrete-energy-general} and group the terms row by row in $U$. This gives exactly \eqref{eq:factorized-energy-term}.
\end{proof}

\subsection{A denominator-invertibility criterion}
The energy identity above is meaningful only when the rational step is well-defined.  The following elementary criterion will be used together with the first-subdiagonal Pad\'e root-location result in Section~\ref{sec:first-subdiagonal}.

\begin{lemma}[Spectral criterion for denominator invertibility~\cite{SunWeiWu2022EnergyLaws}]
\label{lem:denominator-invertibility-criterion}
Let \(q\) be a scalar polynomial with \(q(0)\ne0\), and let \(Z(q)\) be the set of its complex zeros. Here \(\sigma(L)\) denotes the spectrum of the complexification \(L_{\mathbb C}\) of \(L\). If
\begin{equation}
    Z(q)\cap \tau\sigma(L)=\emptyset,
\end{equation}
then \(q(\tau L)\) is invertible.  In particular, if every zero of \(q\) lies in the open right half-plane and \(L\) is seminegative, then \(q(\tau L)\) is invertible for every \(\tau>0\).
\end{lemma}

\begin{proof}
Factor \(q(z)=c\prod_m(z-\zeta_m)\), where \(\zeta_m\in Z(q)\).  If \(\zeta_m\notin\tau\sigma(L)\), then \(\tau L-\zeta_m I\) is invertible by the definition of the spectrum.  Hence \(q(\tau L)=c\prod_m(\tau L-\zeta_m I)\) is invertible.  If all \(\zeta_m\) have positive real parts, then they cannot belong to \(\tau\sigma(L)\). Indeed, the complexified operator is dissipative, so its numerical range is contained in the closed left half-plane; by the standard inclusion of the spectrum in the closure of the numerical range for bounded operators, \(\sigma(L)\) is contained in the closed left half-plane.
\end{proof}

\section{First-subdiagonal Pad\'e approximants and the main energy law}
\label{sec:first-subdiagonal}
We now specialize the general energy framework to the \((s-1,s)\) Pad\'e family.  After recording the coefficient formulas and the denominator invertibility needed for the rational step, we state the explicit Cholesky-type factorization and derive the corresponding energy law.  The proof of the factorization itself is deferred to Section~\ref{sec:first-subdiagonal-proof}.

\subsection{Pad\'e coefficients and well-definedness}
For \(s\ge1\), the \((s-1,s)\) Pad\'e approximant to \(e^z\) is denoted by
\begin{equation}
    \calR_s(z)=\frac{\calP_s(z)}{\calQ_s(z)}.
\end{equation}
In the notation of \eqref{eq:rational-stability}, its coefficients are
\begin{equation}
\begin{aligned}
    \theta_i
    &=
    \frac{(s-1)!}{(2s-1)!}
    \frac{(2s-i-1)!}{i!(s-i-1)!},
    &&0\le i\le s-1,\qquad
    \theta_s=0,  \\
    \vartheta_i
    &=
    (-1)^i
    \frac{s!}{(2s-1)!}
    \frac{(2s-i-1)!}{i!(s-i)!},
    &&0\le i\le s .
\end{aligned}
\label{eq:subdiag-coefficients}
\end{equation}
It is useful to introduce the positive coefficients
\begin{equation}
    \eta_i
    :=
    \frac{s!}{(2s-1)!}
    \frac{(2s-i-1)!}{i!(s-i)!},
    \qquad 0\le i\le s.
    \label{eq:eta-def}
\end{equation}
Then
\begin{equation}
    \vartheta_i=(-1)^i\eta_i,
    \qquad
    \theta_i=\frac{s-i}{s}\eta_i,
    \qquad 0\le i\le s.
    \label{eq:theta-vartheta-eta}
\end{equation}
The second identity also holds for \(i=s\), because \(\theta_s=0\). Consequently, in the matrix range of Lemma~\ref{lem:quoted-general-energy},
\begin{equation}
    \alpha_{ij}
    =
    \left(\frac{(s-i)(s-j)}{s^2}-(-1)^{i+j}\right)\eta_i\eta_j.
    \label{eq:alpha-subdiag}
\end{equation}

The next corollary gives the well-definedness of the rational step.  It is based on the zero-location theorem of Saff--Varga for Pad\'e approximants to the exponential.
\begin{corollary}[Denominator invertibility for the first-subdiagonal family]
\label{cor:subdiag-denominator-invertible}
Let \(L\) be a bounded seminegative operator on the real Hilbert space \(V\).  For every \(s\ge1\) and every \(\tau>0\), the first-subdiagonal denominator \(\calQ_s(\tau L)\) of the \((s-1,s)\) Pad\'e approximant is invertible.
\end{corollary}

\begin{proof}
In the notation of Saff--Varga~\cite{SaffVarga1975}, let \(P_{m,n}\) denote the numerator polynomial of the \((m,n)\)-Pad\'e approximant to the exponential, and let \(Q_{m,n}\) denote the corresponding denominator. Their zero-location theorem implies that all zeros of \(P_{n,\nu}\) lie in the open left half-plane when \(n-\nu=1\). For the first-subdiagonal denominator considered here,
\begin{equation}
    \calQ_s(z)=Q_{s-1,s}(z)=P_{s,s-1}(-z).
\end{equation}
Hence all zeros of \(\calQ_s\) lie in the open right half-plane. The claim follows from Lemma~\ref{lem:denominator-invertibility-criterion}.
\end{proof}

\subsection{Main Cholesky-type factorization}
Define
\begin{equation}
    d_k:=\frac{(k!)^2}{(2k)!(2k+1)!},
    \qquad 0\le k\le s-1,
    \label{eq:dk-def}
\end{equation}
and set
\begin{equation}
    D:=\diag(d_0,d_1,\ldots,d_{s-1}).
\end{equation}
We define an upper triangular matrix \(U=(\mu_{ij})_{i,j=0}^{s-1}\) as follows. For \(0\le i\le j\le s-1\) and \(i\equiv j\pmod 2\), let
\begin{equation}
\mu_{ij}
=
\frac{s!}{(2s)!}
\frac{(2i+1)!}{i!(i+j+1)!}
\frac{(2s+i-j)!}{(s-1-j)!}
\times
\frac{
\left(s-1-\frac{i+j}{2}\right)!
\left(\frac{i+j}{2}\right)!
}{
\left(s-\frac{j-i}{2}\right)!
\left(\frac{j-i}{2}\right)!
} .
\label{eq:mu-even}
\end{equation}
For \(0\le i\le j\le s-1\) and \(i\equiv j+1\pmod 2\), let
\begin{equation}
\mu_{ij}
=
-\frac{2s!}{(2s)!}
\frac{(2i+1)!}{i!(i+j)!}
\frac{(2s+i-j-2)!}{(s-1-j)!}
\times
\frac{
\left(s-1-\frac{i+j+1}{2}\right)!
\left(\frac{i+j-1}{2}\right)!
}{
\left(s-1-\frac{j-i+1}{2}\right)!
\left(\frac{j-i-1}{2}\right)!
} .
\label{eq:mu-odd}
\end{equation}
Finally, set \(\mu_{ij}=0\) for \(i>j\). All factorials in \eqref{eq:mu-even} and \eqref{eq:mu-odd} are taken at nonnegative integers under the stated parity conditions. A direct substitution into \eqref{eq:mu-even} gives \(\mu_{ii}=1\) for every \(0\le i\le s-1\); hence \(U\) is nonsingular.

The following factorization is the main algebraic result of the paper.  Its formula is explicit, but its proof is not a routine Cholesky computation: the entries of \(\Up\) come from alternating Pad\'e coefficient sums, whereas the entries of \(U\) are parity-dependent factorial expressions.  Section~\ref{sec:first-subdiagonal-proof} proves the identity by reducing it to finite telescoping sums.
\begin{theorem}
\label{thm:main-cholesky}
For every integer \(s\ge1\), the coefficient matrix \(\Up=(\gamma_{ij})_{i,j=0}^{s-1}\) associated with the \((s-1,s)\) Pad\'e approximant and defined by \eqref{eq:gamma-def} satisfies
\begin{equation}
    \Up=-U^TDU.
    \label{eq:main-cholesky}
\end{equation}
In particular, \(\Up\) is negative definite.
\end{theorem}

\begin{example}
The first few cases illustrate the structure of the factorization.  For \(s=1\),
\[
    D=(1),\qquad U=(1).
\]
For \(s=2\),
\[
    D=\operatorname{diag}\left(1,\frac1{12}\right),\qquad
    U=
    \begin{pmatrix}
    1&-\frac16\\
    0&1
    \end{pmatrix}.
\]
For \(s=3\),
\[
    D=\operatorname{diag}\left(1,\frac1{12},\frac1{720}\right),\qquad
    U=
    \begin{pmatrix}
    1&-\frac1{10}&\frac1{60}\\
    0&1&-\frac1{10}\\
    0&0&1
    \end{pmatrix}.
\]
These examples provide low-order illustrations of the formula in Theorem~\ref{thm:main-cholesky}; the proof for arbitrary \(s\) is given in Section~\ref{sec:first-subdiagonal-proof}.
\end{example}

\subsection{\(\beta\)-coefficient cancellation}
\begin{theorem}
\label{thm:beta}
For every integer \(s\ge1\), the \((s-1,s)\) Pad\'e approximant satisfies
\begin{equation}
    \beta_k=0\quad(0\le k\le s-1),
    \qquad
    \beta_s=-\left(\frac{(s-1)!}{(2s-1)!}\right)^2.
    \label{eq:beta-zero-last}
\end{equation}
Hence \(B\preceq0\).
\end{theorem}
\begin{proof}[Proof of Theorem~\ref{thm:beta}]
We use the Pad\'e matching property. The \((s-1,s)\) Pad\'e approximant satisfies
\begin{equation}
    \frac{\calP_s(z)}{\calQ_s(z)}=e^z+\mathcal O(z^{2s})
    \qquad (z\to0).
\end{equation}
Since \(\calQ_s(0)=1\), this matching property is equivalent to
\[
    \calP_s(z)=\calQ_s(z)e^z+\mathcal O(z^{2s}).
\]
Replacing \(z\) by \(-z\) and multiplying the two relations gives
\begin{equation}
    \calP_s(z)\calP_s(-z)-\calQ_s(z)\calQ_s(-z)=\mathcal O(z^{2s}).
\end{equation}
The left-hand side is an even polynomial of degree at most \(2s\). Therefore it must be a multiple of \(z^{2s}\). Since \(\deg\calP_s=s-1\) and the leading coefficient of \(\calQ_s\) is
\begin{equation}
    \vartheta_s=(-1)^s\frac{(s-1)!}{(2s-1)!},
\end{equation}
the coefficient of \(z^{2s}\) in
\(\calP_s(z)\calP_s(-z)-\calQ_s(z)\calQ_s(-z)\) is
\begin{equation}
    -(-1)^s\vartheta_s^2
    =
    (-1)^{s+1}\left(\frac{(s-1)!}{(2s-1)!}\right)^2 .
\end{equation}
Thus
\begin{equation}
    \calP_s(z)\calP_s(-z)-\calQ_s(z)\calQ_s(-z)
    =
    (-1)^{s+1}\left(\frac{(s-1)!}{(2s-1)!}\right)^2z^{2s}.
    \label{eq:PPQQ-identity}
\end{equation}
On the other hand, by \eqref{eq:pq-minus-specialization} and the definition \eqref{eq:beta-def}, the coefficient of \(z^{2k}\) in the left-hand side of \eqref{eq:PPQQ-identity} is
\begin{equation}
    \sum_{\ell=\max\{0,2k-s\}}^{\min\{2k,s\}}
    (-1)^{2k-\ell}\alpha_{\ell,2k-\ell}
    =
    (-1)^k\beta_k .
    \label{eq:z2k-coefficient-beta}
\end{equation}
Comparing coefficients in \eqref{eq:PPQQ-identity} and using \eqref{eq:z2k-coefficient-beta}, we obtain
\[
    \beta_0=\beta_1=\cdots=\beta_{s-1}=0
\]
and
\begin{equation}
    (-1)^s\beta_s
    =
    (-1)^{s+1}\left(\frac{(s-1)!}{(2s-1)!}\right)^2 .
\end{equation}
Therefore
\begin{equation}
    \beta_s
    =
    -\left(\frac{(s-1)!}{(2s-1)!}\right)^2,
\end{equation}
which proves the theorem.
\end{proof}

\subsection{Discrete energy law and unconditional contractivity}
\begin{theorem}
\label{thm:energy-law}
For every integer \(s\ge1\), consider the time-stepping scheme \(u^{n+1}=\calR_s(\tau L)u^n\), obtained by applying the \((s-1,s)\) Pad\'e approximant to the linear seminegative system \eqref{eq:linear-system}. Let \(w^n=Q^{-1}u^n\), where \(Q=\calQ_s(\tau L)\). Then
\begin{equation}
    \|u^{n+1}\|^2-\|u^n\|^2
    =
    -\left(\frac{(s-1)!}{(2s-1)!}\right)^2
    \tau^{2s}\|L^s w^n\|^2
    -
    \sum_{k=0}^{s-1}
    d_k\tau^{2k+1}
    \sn{L^k u^{(k)}}^2,
\label{eq:main-energy-law}
\end{equation}
where
\begin{equation}
    u^{(k)}
    :=
    \sum_{j=k}^{s-1}
    \mu_{kj}(\tau L)^{j-k}Q^{-1}u^n .
    \label{eq:uk-def}
\end{equation}
Consequently,
\begin{equation}
    \|u^{n+1}\|\le \|u^n\|,
    \qquad \forall \tau>0.
    \label{eq:strong-stability}
\end{equation}
Since the inequality holds for every initial value \(u^n\in V\), the one-step operator is contractive:
\begin{equation}
    \|\calR_s(\tau L)\|_{\mathcal L(V)}\le 1,
    \qquad \forall \tau>0.
\end{equation}
Thus the resulting \((s-1,s)\) Pad\'e time-stepping method is unconditionally strongly stable for the linear seminegative systems with bounded operators considered here.
\end{theorem}

\begin{proof}[Proof of Theorem~\ref{thm:energy-law}]
By Corollary~\ref{cor:subdiag-denominator-invertible}, the Pad\'e step is well-defined for every $\tau>0$. Substituting Theorem~\ref{thm:beta} into the general identity \eqref{eq:discrete-energy-general} gives
\begin{equation}
\sum_{k=0}^{s}\beta_k\tau^{2k}\|L^k w^n\|^2
=
-
\left(\frac{(s-1)!}{(2s-1)!}\right)^2
\tau^{2s}\|L^s w^n\|^2.
\end{equation}
By Theorem~\ref{thm:main-cholesky} and the decomposition $\Up=-U^TDU$,
\begin{equation}
\sum_{i=0}^{s-1}\sum_{j=0}^{s-1}
\gamma_{ij}\tau^{i+j+1}[L^i w^n,L^j w^n]_L
=
-
\sum_{k=0}^{s-1}d_k\tau^{2k+1}\sn{L^k u^{(k)}}^2,
\end{equation}
where $u^{(k)}$ is defined by \eqref{eq:uk-def}. Combining the two identities proves \eqref{eq:main-energy-law}. Since every term on the right-hand side is nonpositive, \eqref{eq:strong-stability} follows immediately.
\end{proof}

\begin{remark}[Comparison with algebraic stability]
The contractivity conclusion in Theorem~\ref{thm:energy-law} is consistent with the algebraic stability of the corresponding Radau IIA methods~\cite{HairerWanner1981AlgebraicallyStable,HairerWanner1996SolvingODEII}. The algebraic-stability framework applies to a broad class of nonlinear dissipative problems through a stage-based energy relation. For the linear problem considered here, \eqref{eq:main-energy-law} provides a more refined characterization by making the orders of the dissipation terms explicit. In particular, when $L$ is skew-adjoint, the semi-inner-product terms vanish and the energy law reduces exactly to the $\tau^{2s}$ norm-dissipation term, thereby exposing a high-order structure that is not apparent in the standard stage representation. In addition, Theorem~\ref{thm:main-cholesky} gives a unified closed-form factorization of the energy coefficient matrix valid for every $s$, rather than a method-specific decomposition at a fixed order.
\end{remark}

\section{Proof of the Cholesky-type factorization}
\label{sec:first-subdiagonal-proof}
This section proves the Cholesky-type factorization in Theorem~\ref{thm:main-cholesky}. The difficulty is that the entries of \(\Up\) are alternating sums of Pad\'e coefficients, whereas the proposed triangular factor \(U\) is given by parity-dependent factorial expressions. We reduce the matrix identity to scalar rational identities and prove them by finite telescoping arguments. The longer algebraic verifications are collected in Appendix~\ref{app:algebraic-verifications}.

\subsection{Reduction to an extended scalar identity}

We now carry out this reduction. The proof follows the rational-extension strategy used for diagonal Pad\'e energy laws in~\cite{SunWeiWu2022EnergyLaws}, but the scalar identities and parity structure are specific to the first-subdiagonal family.

For $i\ge0$, define
\begin{equation}
    \eta_i^{(s)}
    :=
    \frac{1}{i!}
    \frac{s(s-1)\cdots(s-i+1)}{(2s-1)(2s-2)\cdots(2s-i)},
    \qquad \eta_0^{(s)}:=1,
    \label{eq:eta-rational}
\end{equation}
and
\begin{equation}
    \theta_i^{(s)}:=\frac{s-i}{s}\eta_i^{(s)}.
    \label{eq:theta-rational}
\end{equation}
For positive integer $s$, these definitions agree with \eqref{eq:eta-def} and \eqref{eq:theta-vartheta-eta} whenever $0\le i\le s$. Define
\begin{equation}
    \alpha_{ij}^{(s)}
    :=
    \theta_i^{(s)}\theta_j^{(s)}
    -(-1)^{i+j}\eta_i^{(s)}\eta_j^{(s)}
    =
    \left(\frac{(s-i)(s-j)}{s^2}-(-1)^{i+j}\right)
    \eta_i^{(s)}\eta_j^{(s)}.
    \label{eq:alpha-rational}
\end{equation}
For $p,q\ge0$, set
\begin{equation}
    \gamma_{pq}^{(s)}
    :=
    \sum_{\ell=0}^{\min\{p,q\}}
    (-1)^{\min\{p,q\}+1-\ell}
    \alpha_{\ell,p+q+1-\ell}^{(s)}.
    \label{eq:gamma-rational}
\end{equation}
When $s$ is a positive integer and $0\le p,q\le s-1$, \eqref{eq:gamma-rational} reduces to \eqref{eq:gamma-def}.  Indeed, every index of the form $p+q+1-\ell$ is at most $2s-1$; if such an index is larger than $s$, then the numerator of \(\eta_{p+q+1-\ell}^{(s)}\) contains a zero factor whereas no denominator factor vanishes, and the corresponding term is zero.

Let
\begin{equation}
    \hR:=\{x\in\R:2x\notin\Z\}.
\end{equation}
For $i,j\ge1$ and $s\in\hR$, define the auxiliary quantities $\nu_{ij}^{(s)}$ by
\begin{align}
\nu_{2i-1,2j}^{(s)}
&=
\frac{2\sqrt{4i-3}}{s}
\frac{
\left(s-j+\frac12\right)_{i-1}(-j)_i
}{
(j-s+1)_{i-1}
\left(j+\frac12\right)_{i-1}
}
\eta_{2j}^{(s)},
\label{eq:nu-oe}
\\
\nu_{2i,2j-1}^{(s)}
&=
-\frac{2\sqrt{4i-1}}{s}
\frac{
\left(s-j+\frac12\right)_i(1-j)_i
}{
(j-s)_i
\left(j+\frac12\right)_{i-1}
}
\eta_{2j-1}^{(s)},
\label{eq:nu-eo}
\\
\nu_{2i,2j}^{(s)}
&=
-\frac{2\sqrt{4i-1}}{s}
\frac{
\left(s+\frac12-j\right)_i(-j)_i
}{
(j-s+1)_{i-1}
\left(j+\frac12\right)_i
}
\eta_{2j}^{(s)},
\label{eq:nu-ee}
\\
\nu_{2i-1,2j-1}^{(s)}
&=
\frac{2\sqrt{4i-3}}{s}
\frac{
\left(s+\frac12-j\right)_i(1-j)_{i-1}
}{
(j-s)_{i-1}
\left(j+\frac12\right)_{i-1}
}
\eta_{2j-1}^{(s)}.
\label{eq:nu-oo}
\end{align}
Here $(x)_n$ denotes the Pochhammer symbol,
\begin{equation}
    (x)_0:=1,
    \qquad
    (x)_n:=x(x+1)\cdots(x+n-1),\quad n\ge1.
\end{equation}
The following elementary identities will be used as rational identities. When factorial notation is applied to noninteger arguments, it is understood through the Gamma function away from its poles.
\begin{lemma}[Basic Pochhammer reductions]
\label{lem:pochhammer-basic}
The following identities hold away from the poles, and hence as rational identities after simplification:
\begin{align}
    (x+n)!&=x!(x+1)_n,\label{eq:pochhammer-shift}\\
    (x)_n&=2^n
    \left(\frac{x}{2}\right)_{\lceil n/2\rceil}
    \left(\frac{x+1}{2}\right)_{\lfloor n/2\rfloor},
    \label{eq:pochhammer-duplication}\\
    \frac{(x+i)!}{(x-j)!}&=(-1)^j(-x)_j(x+1)_i.
    \label{eq:pochhammer-ratio}
\end{align}
\end{lemma}

The factors $(-j)_i$ or $(1-j)_i$ imply $\nu_{ij}^{(s)}=0$ whenever $i>j$. The formulas \eqref{eq:nu-oe}--\eqref{eq:nu-oo} are first used for $s\in\hR$, where the displayed Pochhammer denominators do not vanish. They will later be specialized to positive integer $s$ through rational continuation.

\begin{lemma}[Connection with the Cholesky factor]
\label{lem:nu-mu-connection}
For every positive integer $s$ and $1\le i,j\le s$, the rational continuation of $\nu_{ij}^{(s)}$ to this integer value satisfies
\begin{equation}
    \nu_{ij}^{(s)}
    =
    \sqrt{d_{i-1}}\,\mu_{i-1,j-1}.
    \label{eq:nu-mu-relation}
\end{equation}
\end{lemma}

\begin{proof}
The four parity cases give the following identities after rational continuation to positive integer values of \(s\):
\begin{align}
\nu_{2i-1,2j}^{(s)}&=\sqrt{d_{2i-2}}\,\mu_{2i-2,2j-1},
\label{eq:nu-mu-oe}\\
\nu_{2i,2j-1}^{(s)}&=\sqrt{d_{2i-1}}\,\mu_{2i-1,2j-2},
\label{eq:nu-mu-eo}\\
\nu_{2i,2j}^{(s)}&=\sqrt{d_{2i-1}}\,\mu_{2i-1,2j-1},
\label{eq:nu-mu-ee}\\
\nu_{2i-1,2j-1}^{(s)}&=\sqrt{d_{2i-2}}\,\mu_{2i-2,2j-2}.
\label{eq:nu-mu-oo}
\end{align}
They are obtained by simplifying the four formulas \eqref{eq:nu-oe}--\eqref{eq:nu-oo} with the two definitions \eqref{eq:mu-even}--\eqref{eq:mu-odd}.  The necessary Pochhammer-to-factorial reductions and the treatment of removable singularities are recorded in Appendix~\ref{app:nu-mu-verification}.  These four identities are exactly \eqref{eq:nu-mu-relation} for the four possible parities of \((i,j)\).
\end{proof}

The scalar identity needed for the matrix factorization is the following.
\begin{theorem}[Extended scalar identity]
\label{thm:extended-scalar}
For all $p,q\in\Npos$ and all $s\in\hR$,
\begin{equation}
    \gamma_{p-1,q-1}^{(s)}
    +
    \sum_{i=1}^{\infty}
    \nu_{i,p}^{(s)}\nu_{i,q}^{(s)}
    =0.
    \label{eq:scalar-identity-goal}
\end{equation}
The infinite sum is finite because $\nu_{i,p}^{(s)}=0$ for $i>p$.
\end{theorem}

We postpone the proof of Theorem~\ref{thm:extended-scalar} to the next subsections and first complete the proof of Theorem~\ref{thm:main-cholesky} assuming it.

\begin{proof}[Proof of Theorem~\ref{thm:main-cholesky} assuming Theorem~\ref{thm:extended-scalar}]
For fixed $p,q\in\Npos$, define
\begin{equation}
    F_{pq}(s)
    :=
    \gamma_{p-1,q-1}^{(s)}
    +
    \sum_{i=1}^{\infty}
    \nu_{i,p}^{(s)}\nu_{i,q}^{(s)}.
    \label{eq:Fpq-def}
\end{equation}
The sum in \eqref{eq:Fpq-def} is finite. Moreover, the square-root factors in $\nu_{i,p}^{(s)}\nu_{i,q}^{(s)}$ occur only through products with the same first index $i$, so $F_{pq}(s)$ is a rational function of $s$ on its domain of definition. By Theorem~\ref{thm:extended-scalar}, this rational function is zero for every $s\in\hR$. Since \(\hR\) is infinite and contains accumulation points away from the finite set of poles of \(F_{pq}\), the function \(F_{pq}\) must be the zero rational function. Hence the identity extends across removable singularities to positive integer values of $s$.

Now let $s$ be a positive integer and $1\le p,q\le s$.  Evaluating the continued identity at this $s$ and using Lemma~\ref{lem:nu-mu-connection} gives
\begin{equation}
    \gamma_{p-1,q-1}
    +
    \sum_{i=1}^{s}
    d_{i-1}\mu_{i-1,p-1}\mu_{i-1,q-1}
    =0.
\end{equation}
This is precisely the $(p,q)$ entry of $\Up+U^TDU=0$. Therefore $\Up=-U^TDU$. Since $D$ has strictly positive diagonal entries and $U$ is upper triangular with diagonal entries $\mu_{ii}=1$, $\Up$ is negative definite. This proves Theorem~\ref{thm:main-cholesky}.
\end{proof}

\subsection{Technical ingredients for the extended scalar identity}
\subsubsection{Telescoping identities for the four parity classes}
The proof requires four finite summation identities. They have the same structure: a hypergeometric term is written as a telescoping difference. We state them together to make the later proof concise.

For $n\in\Nzero$, define $H_n^{(r)}$, $\Psi_n^{(r)}$, and $\Phi_n^{(r)}$ as follows.

\paragraph{Type I.}
For $p,q\in\Npos$, set
\begin{equation}
H_n^{(1)}=
\frac{
\left(s+\frac32-p\right)_n(1-p)_n
\left(s+\frac12-q\right)_n(-q)_n
}{
(p-s+1)_n\left(p+\frac32\right)_n
(q-s+1)_n\left(q+\frac32\right)_n
}.
\label{eq:H1}
\end{equation}
Let
\begin{align}
\cC_{1,n}^{(1)}&=(4n+3)(1+s-2p)(q-n)(1+2s+2n-2q),
\label{eq:C11}\\
\cC_{2,n}^{(1)}&=(4n+1)(s-2q)(1+2p+2n)(s-p-n),
\label{eq:C21}\\
\cC_{3,n}^{(1)}&=(n+p-s)(1+2p+2n)(n+q-s)(1+2q+2n),
\label{eq:C31}
\end{align}
and
\begin{equation}
\resizebox{0.85\linewidth}{!}{$\displaystyle
\Psi_n^{(1)}=
H_n^{(1)}
\frac{\cC_{1,n}^{(1)}+\cC_{2,n}^{(1)}}{(s-p)(1+2p)(s-q)(1+2q)},
\qquad
\Phi_n^{(1)}=
H_n^{(1)}
\frac{\cC_{3,n}^{(1)}}{(s-p)(1+2p)(s-q)(1+2q)}.
$}
\label{eq:PsiPhi1}
\end{equation}

\paragraph{Type II.}
Set
\begin{equation}
H_n^{(2)}=
\frac{
\left(s-p+\frac12\right)_n(1-p)_n
\left(s-q+\frac12\right)_n(1-q)_n
}{
(p-s+1)_n\left(p+\frac12\right)_n
(q-s+2)_n\left(q+\frac32\right)_n
}.
\label{eq:H2}
\end{equation}
Let
\begin{align}
\cC_{1,n}^{(2)}&=(4n+3)(s-2q)(2p-1-2n-2s)(n+1-p),
\label{eq:C12}\\
\cC_{2,n}^{(2)}&=(4n+1)(s-2p+1)(s-q-n-1)(2q+2n+1),
\label{eq:C22}\\
\cC_{3,n}^{(2)}&=(p-s+n)(2p+2n-1)(q-s+n+1)(2q+2n+1),
\label{eq:C32}
\end{align}
and
\begin{equation}
\resizebox{0.85\linewidth}{!}{$\displaystyle
\Psi_n^{(2)}=
H_n^{(2)}
\frac{\cC_{1,n}^{(2)}+\cC_{2,n}^{(2)}}{(2q+1)(2p-1)(p-s)(q-s+1)},
\qquad
\Phi_n^{(2)}=
H_n^{(2)}
\frac{\cC_{3,n}^{(2)}}{(2q+1)(2p-1)(p-s)(q-s+1)}.
$}
\label{eq:PsiPhi2}
\end{equation}

\paragraph{Type III.}
Set
\begin{equation}
H_n^{(3)}=
\frac{
\left(s+\frac12-p\right)_n(-p)_n
\left(s+\frac12-q\right)_n(1-q)_n
}{
(p-s+1)_n\left(p+\frac32\right)_n
(q-s+2)_n\left(q+\frac32\right)_n
}.
\label{eq:H3}
\end{equation}
Let
\begin{align}
\cC_{1,n}^{(3)}&=(4n+3)(s-2q)(p-n)(2s-2p+2n+1),
\label{eq:C13}\\
\cC_{2,n}^{(3)}&=(4n+1)(s-2p)(s-q-n-1)(2q+2n+1),
\label{eq:C23}\\
\cC_{3,n}^{(3)}&=(p-s+n)(2p+2n+1)(q-s+n+1)(2q+2n+1),
\label{eq:C33}
\end{align}
and
\begin{equation}
\resizebox{0.85\linewidth}{!}{$\displaystyle
\Psi_n^{(3)}=
H_n^{(3)}
\frac{\cC_{1,n}^{(3)}+\cC_{2,n}^{(3)}}{(2p+1)(p-s)(q-s+1)(2q+1)},
\qquad
\Phi_n^{(3)}=
H_n^{(3)}
\frac{\cC_{3,n}^{(3)}}{(2p+1)(p-s)(q-s+1)(2q+1)}.
$}
\label{eq:PsiPhi3}
\end{equation}

\paragraph{Type IV.}
Set
\begin{equation}
H_n^{(4)}=
\frac{
\left(s-p+\frac12\right)_n(1-p)_n
\left(s+\frac32-q\right)_n(1-q)_n
}{
(p-s+1)_n\left(p+\frac12\right)_n
(q-s+1)_n\left(q+\frac32\right)_n
}.
\label{eq:H4}
\end{equation}
Let
\begin{align}
\cC_{1,n}^{(4)}&=(4n+3)(s-2q+1)(2s+2n-2p+1)(p-n-1),
\label{eq:C14}\\
\cC_{2,n}^{(4)}&=(4n+1)(s-2p+1)(s-q-n)(2q+2n+1),
\label{eq:C24}\\
\cC_{3,n}^{(4)}&=(p-s+n)(2p+2n-1)(q-s+n)(2q+2n+1),
\label{eq:C34}
\end{align}
and
\begin{equation}
\resizebox{0.85\linewidth}{!}{$\displaystyle
\Psi_n^{(4)}=
H_n^{(4)}
\frac{\cC_{1,n}^{(4)}+\cC_{2,n}^{(4)}}{(2p-1)(p-s)(q-s)(2q+1)},
\qquad
\Phi_n^{(4)}=
H_n^{(4)}
\frac{\cC_{3,n}^{(4)}}{(2p-1)(p-s)(q-s)(2q+1)}.
$}
\label{eq:PsiPhi4}
\end{equation}
When the parameter dependence must be displayed explicitly, we write
\(\Psi_n^{(\rho)}(s;p,q)\) and \(\Phi_n^{(\rho)}(s;p,q)\), \(\rho=1,2,3,4\), for the Type~\(\rho\) quantities obtained from the above formulas with the displayed parameters \(s,p,q\).

\begin{lemma}[Telescoping identities]
\label{lem:telescoping-certificates}
For each $r=1,2,3,4$, the following identity holds whenever the corresponding quantities are defined:
\begin{equation}
    \Psi_n^{(r)}=\Phi_n^{(r)}-\Phi_{n+1}^{(r)},
    \qquad n\ge0.
    \label{eq:telescoping-certificate}
\end{equation}
Moreover,
\begin{align*}
\Phi_0^{(1)}&=1, & \Phi_n^{(1)}&=0 &&\text{for }n\ge p,\\
\Phi_0^{(2)}&=1, & \Phi_n^{(2)}&=0 &&\text{for }n\ge p,\\
\Phi_0^{(3)}&=1, & \Phi_n^{(3)}&=0 &&\text{for }n\ge q,\\
\Phi_0^{(4)}&=1, & \Phi_n^{(4)}&=0 &&\text{for }n\ge q.
\end{align*}
Consequently,
\begin{equation}
    \sum_{n=0}^{\infty}\Psi_n^{(r)}=1,
    \qquad r=1,2,3,4,
    \label{eq:Psi-sum-one}
\end{equation}
where each infinite series is in fact finite.
\end{lemma}

\begin{proof}
For each type, let \(\rho_n^{(r)}\) denote the explicit rational factor listed in Appendix~\ref{app:telescoping-verification}. Applying \((x)_{n+1}=(x)_n(x+n)\) to the Pochhammer factors in \(H_n^{(r)}\) gives
\begin{equation*}
    H_{n+1}^{(r)}=\rho_n^{(r)}H_n^{(r)}.
\end{equation*}
Let \(\Delta^{(r)}\) denote the common denominator in the definitions of \(\Psi_n^{(r)}\) and \(\Phi_n^{(r)}\). It therefore suffices to verify
\begin{equation}
    \cC_{3,n+1}^{(r)}\rho_n^{(r)}-\cC_{3,n}^{(r)}
    =-\cC_{1,n}^{(r)}-\cC_{2,n}^{(r)}.
    \label{eq:C-polynomial-identity}
\end{equation}
For \(r=1,2,3,4\), this is a finite polynomial identity after clearing the displayed denominators. The four rational factors \(\rho_n^{(r)}\) and the corresponding polynomial checks are given in Appendix~\ref{app:telescoping-verification}. Multiplying \eqref{eq:C-polynomial-identity} by \(H_n^{(r)}/\Delta^{(r)}\) and using the preceding identity gives \eqref{eq:telescoping-certificate}.

The initial values \(\Phi_0^{(r)}=1\) follow immediately from \(H_0^{(r)}=1\) and the definitions of \(\cC_{3,0}^{(r)}\). The termination follows from the factors \((1-p)_n\) in Types I and II, and from the factors \((1-q)_n\) in Types III and IV. Summing \eqref{eq:telescoping-certificate} over \(n\) gives \eqref{eq:Psi-sum-one}.
\end{proof}

\subsubsection{Product identities for adjacent Cholesky entries}
The next lemma connects the quantities $\nu_{ij}^{(s)}$ with the terminating sums above.

\begin{lemma}[Product reductions]
For all positive integers $i,p,q$ and $s\in\hR$, the following identities hold:
\begin{align}
\nu_{2i-1,2p-1}^{(s)}\nu_{2i-1,2q+1}^{(s)}
+\nu_{2i,2p}^{(s)}\nu_{2i,2q}^{(s)}
&=
\frac{(s-q)(2s+1-2p)}{s^2}
\Psi_{i-1}^{(1)}(s;p,q)
\eta_{2p-1}^{(s)}\eta_{2q}^{(s)},
\label{eq:prod1}\\
\nu_{2i,2p-1}^{(s)}\nu_{2i,2q+1}^{(s)}
+\nu_{2i-1,2p}^{(s)}\nu_{2i-1,2q}^{(s)}
&=
\frac{(2p-1)q}{s^2}
\Psi_{i-1}^{(2)}(s;p,q)
\eta_{2p-1}^{(s)}\eta_{2q}^{(s)},
\label{eq:prod2}\\
\nu_{2i,2q+1}^{(s)}\nu_{2i,2p}^{(s)}
+\nu_{2i-1,2p+1}^{(s)}\nu_{2i-1,2q}^{(s)}
&=
\frac{2q(p-s)}{s^2}
\Psi_{i-1}^{(3)}(s;p,q)
\eta_{2p}^{(s)}\eta_{2q}^{(s)},
\label{eq:prod3}\\
\nu_{2i-1,2p}^{(s)}\nu_{2i-1,2q-1}^{(s)}
+\nu_{2i,2p-1}^{(s)}\nu_{2i,2q}^{(s)}
&=
\frac{(1-2p)(2s-2q+1)}{2s^2}
\Psi_{i-1}^{(4)}(s;p,q)
\eta_{2p-1}^{(s)}\eta_{2q-1}^{(s)}.
\label{eq:prod4}
\end{align}
\end{lemma}

\begin{proof}
We prove \eqref{eq:prod1} in detail.  The remaining three reductions are obtained by the same direct substitution of \eqref{eq:nu-oe}--\eqref{eq:nu-oo}; their common factors are listed in Appendix~\ref{app:product-reductions}.

For \eqref{eq:prod1}, using \eqref{eq:nu-oo} gives
\begin{equation*}
\nu_{2i-1,2p-1}^{(s)}\nu_{2i-1,2q+1}^{(s)}
=
\frac{(s-q)(2s+1-2p)}{s^2}
\eta_{2p-1}^{(s)}\eta_{2q}^{(s)}
H_{i-1}^{(1)}
\frac{\cC_{2,i-1}^{(1)}}{(s-p)(1+2p)(s-q)(1+2q)} .
\end{equation*}
Similarly, using \eqref{eq:nu-ee} gives
\begin{equation*}
\nu_{2i,2p}^{(s)}\nu_{2i,2q}^{(s)}
=
\frac{(s-q)(2s+1-2p)}{s^2}
\eta_{2p-1}^{(s)}\eta_{2q}^{(s)}
H_{i-1}^{(1)}
\frac{\cC_{1,i-1}^{(1)}}{(s-p)(1+2p)(s-q)(1+2q)} .
\end{equation*}
Adding these two identities and using the definition of \(\Psi_{i-1}^{(1)}\) in \eqref{eq:PsiPhi1} proves \eqref{eq:prod1}.  The appendix gives the analogous common-factor forms for \eqref{eq:prod2}--\eqref{eq:prod4}, which prove the remaining identities in the same way.
\end{proof}

\subsubsection{Adjacent-index summation formula}
\begin{proposition}[Adjacent summation identity]\label{prop:adjacent}
For all $p,q\in\Npos$ and $s\in\hR$,
\begin{equation}
\sum_{i=1}^{\infty}\nu_{i,p}^{(s)}\nu_{i,q+1}^{(s)}
+
\sum_{i=1}^{\infty}\nu_{i,p+1}^{(s)}\nu_{i,q}^{(s)}
=
\alpha_{pq}^{(s)}.
\label{eq:adjacent-identity}
\end{equation}
\end{proposition}

\begin{proof}
All sums are finite. We split the proof according to the parities of $p$ and $q$.

First suppose $p=2a-1$ and $q=2b$. By \eqref{eq:prod1}, \eqref{eq:prod2}, and \eqref{eq:Psi-sum-one},
\begin{align*}
&\sum_{i=1}^{\infty}\nu_{i,2a-1}^{(s)}\nu_{i,2b+1}^{(s)}
+
\sum_{i=1}^{\infty}\nu_{i,2a}^{(s)}\nu_{i,2b}^{(s)} \\
&\quad=
\left(
\frac{(s-b)(2s+1-2a)}{s^2}
+
\frac{(2a-1)b}{s^2}
\right)
\eta_{2a-1}^{(s)}\eta_{2b}^{(s)} \\
&\quad=
\left(
\frac{(s-2a+1)(s-2b)}{s^2}+1
\right)
\eta_{2a-1}^{(s)}\eta_{2b}^{(s)} \\
&\quad=
\theta_{2a-1}^{(s)}\theta_{2b}^{(s)}+
\eta_{2a-1}^{(s)}\eta_{2b}^{(s)}
=\alpha_{2a-1,2b}^{(s)}.
\end{align*}
The case where $p$ is even and $q$ is odd follows by symmetry.

Next suppose $p=2a$ and $q=2b$. Applying \eqref{eq:prod3} and also \eqref{eq:prod3} with $a$ and $b$ interchanged gives
\begin{align*}
&\sum_{i=1}^{\infty}\nu_{i,2a}^{(s)}\nu_{i,2b+1}^{(s)}
+
\sum_{i=1}^{\infty}\nu_{i,2a+1}^{(s)}\nu_{i,2b}^{(s)} \\
&\quad=
\left(
\frac{2b(a-s)}{s^2}+\frac{2a(b-s)}{s^2}
\right)
\eta_{2a}^{(s)}\eta_{2b}^{(s)} \\
&\quad=
\left(
\frac{(s-2a)(s-2b)}{s^2}-1
\right)
\eta_{2a}^{(s)}\eta_{2b}^{(s)} \\
&\quad=
\theta_{2a}^{(s)}\theta_{2b}^{(s)}-
\eta_{2a}^{(s)}\eta_{2b}^{(s)}
=\alpha_{2a,2b}^{(s)}.
\end{align*}
Finally suppose $p=2a-1$ and $q=2b-1$. By \eqref{eq:prod4} and the same identity with $a$ and $b$ interchanged,
\begin{align*}
&\sum_{i=1}^{\infty}\nu_{i,2a-1}^{(s)}\nu_{i,2b}^{(s)}
+
\sum_{i=1}^{\infty}\nu_{i,2a}^{(s)}\nu_{i,2b-1}^{(s)} \\
&\quad=
\left(
\frac{(1-2a)(2s-2b+1)}{2s^2}
+
\frac{(1-2b)(2s-2a+1)}{2s^2}
\right)
\eta_{2a-1}^{(s)}\eta_{2b-1}^{(s)} \\
&\quad=
\left(
\frac{(s-2a+1)(s-2b+1)}{s^2}-1
\right)
\eta_{2a-1}^{(s)}\eta_{2b-1}^{(s)} \\
&\quad=
\theta_{2a-1}^{(s)}\theta_{2b-1}^{(s)}-
\eta_{2a-1}^{(s)}\eta_{2b-1}^{(s)}
=\alpha_{2a-1,2b-1}^{(s)}.
\end{align*}
This covers all parity cases and proves the proposition.
\end{proof}

\subsection{Proof of Theorem~\ref{thm:extended-scalar}}
\begin{proof}
Since both sides of \eqref{eq:scalar-identity-goal} are symmetric in $p$ and $q$, assume $1\le p\le q$.

We first treat the boundary case $p=1$. Because $\nu_{i,1}^{(s)}=0$ for $i>1$ and $\nu_{1,1}^{(s)}=1$, we have
\begin{equation}
    \sum_{i=1}^{\infty}\nu_{i,1}^{(s)}\nu_{i,q}^{(s)}=\nu_{1,q}^{(s)}.
\end{equation}
If $q=2b$, then \eqref{eq:nu-oe} gives
\begin{equation}
    \nu_{1,2b}^{(s)}=-\frac{2b}{s}\eta_{2b}^{(s)}
    =\left(\frac{s-2b}{s}-1\right)\eta_{2b}^{(s)}
    =\alpha_{0,2b}^{(s)}.
\end{equation}
If $q=2b-1$, then \eqref{eq:nu-oo} gives
\begin{equation}
    \nu_{1,2b-1}^{(s)}=\frac{2s-2b+1}{s}\eta_{2b-1}^{(s)}
    =\left(\frac{s-(2b-1)}{s}+1\right)\eta_{2b-1}^{(s)}
    =\alpha_{0,2b-1}^{(s)}.
\end{equation}
Since $\gamma_{0,q-1}^{(s)}=-\alpha_{0,q}^{(s)}$, the identity \eqref{eq:scalar-identity-goal} follows for $p=1$.

Now let $p>1$. We use the alternating decomposition
\begin{equation}
\begin{aligned}
\sum_{i=1}^{\infty}\nu_{i,p}^{(s)}\nu_{i,q}^{(s)}
&=
\sum_{k=1}^{p-1}(-1)^{k-1}
\left(
\sum_{i=1}^{\infty}\nu_{i,p-k+1}^{(s)}\nu_{i,q+k-1}^{(s)}
\right.\\
&\qquad\qquad\qquad\left.
+
\sum_{i=1}^{\infty}\nu_{i,p-k}^{(s)}\nu_{i,q+k}^{(s)}
\right)
+
(-1)^{p-1}
\sum_{i=1}^{\infty}\nu_{i,1}^{(s)}\nu_{i,p+q-1}^{(s)}.
\end{aligned}
\label{eq:alternating-decomposition}
\end{equation}
This identity is obtained by expanding the right-hand side; all intermediate sums cancel in pairs, leaving precisely the term on the left-hand side.
By Proposition \ref{prop:adjacent},
\begin{equation}
\sum_{i=1}^{\infty}\nu_{i,p-k+1}^{(s)}\nu_{i,q+k-1}^{(s)}
+
\sum_{i=1}^{\infty}\nu_{i,p-k}^{(s)}\nu_{i,q+k}^{(s)}
=
\alpha_{p-k,q+k-1}^{(s)}.
\end{equation}
The boundary case already proved gives
\begin{equation}
\sum_{i=1}^{\infty}\nu_{i,1}^{(s)}\nu_{i,p+q-1}^{(s)}
=
\alpha_{0,p+q-1}^{(s)}.
\end{equation}
Substituting these two identities into \eqref{eq:alternating-decomposition} yields
\begin{equation}
\begin{aligned}
\sum_{i=1}^{\infty}\nu_{i,p}^{(s)}\nu_{i,q}^{(s)}
&=
\sum_{k=1}^{p}(-1)^{k-1}
\alpha_{p-k,q+k-1}^{(s)} \\
&=
\sum_{j=0}^{p-1}(-1)^{p-j-1}
\alpha_{j,p+q-1-j}^{(s)} \\
&=-
\sum_{j=0}^{p-1}(-1)^{p-j}
\alpha_{j,p+q-1-j}^{(s)}.
\end{aligned}
\end{equation}
Because $p\le q$, the last sum is exactly $-\gamma_{p-1,q-1}^{(s)}$ by \eqref{eq:gamma-rational}. This proves \eqref{eq:scalar-identity-goal} and hence Theorem~\ref{thm:extended-scalar}.
\end{proof}

\section{Numerical experiments}
\label{sec:numerical-experiments}
This section presents three numerical tests, adapted from the test problems in~\cite{SunWeiWu2022EnergyLaws}, to illustrate the first-subdiagonal energy law.  The purpose is twofold: to observe the expected order \(2s-1\) of the \((s-1,s)\) Pad\'e approximants and to verify the stepwise dissipation identity predicted by Theorem~\ref{thm:energy-law}.  All computations are performed in double precision.  In each test we compare the numerical one-step energy dissipation
\begin{equation}
    \mathcal D_{\rm num}^n:=\|u^n\|^2-\|u^{n+1}\|^2
    \label{eq:numerical-dissipation}
\end{equation}
with the theoretical dissipation predicted by the energy law \eqref{eq:main-energy-law}, namely
\begin{equation}
\begin{aligned}
    \mathcal D_{\rm law}^n
    &:={}
    \left(\frac{(s-1)!}{(2s-1)!}\right)^2
    \tau^{2s}\|L^s w^n\|^2
    +
    \sum_{k=0}^{s-1} d_k\tau^{2k+1}\sn{L^k u^{(k)}}^2 .
\end{aligned}
    \label{eq:law-dissipation}
\end{equation}
The identity \eqref{eq:main-energy-law} predicts \(\mathcal D_{\rm num}^n=\mathcal D_{\rm law}^n\), up to round-off error.  In exact arithmetic, it also gives \(\mathcal D_{\rm num}^n\ge0\).  In the plots below, the dissipation generated by the step from \(u^n\) to \(u^{n+1}\) is displayed at the right endpoint \(t_{n+1}\).  All finite-dimensional computations use the Euclidean norm. In Examples~\ref{ex:numerical-p1-convection} and~\ref{ex:numerical-p0-dispersion}, under the adopted modal normalization, its square differs from the standard discrete \(L^2\)-norm squared only by the fixed factor \(\Delta x\). This scaling does not affect contractivity or the comparison between \(\mathcal D_{\rm num}^n\) and \(\mathcal D_{\rm law}^n\).

\begin{example}
\label{ex:numerical-3by3}
We first consider the linear seminegative system from~\cite{SunShu2017RK4}
\begin{equation}
    \frac{\dd}{\dd t}u=Lu,\qquad
    u(t)\in \mathbb R^3,\qquad
    L=-\begin{pmatrix}
    1&2&2\\
    0&1&2\\
    0&0&1
    \end{pmatrix}.
    \label{eq:example-3by3-system}
\end{equation}
The initial condition is chosen as
\begin{equation}
    u(0)=(0.9134,\ 0.2785,\ 0.5469)^T,
\end{equation}
and the system is solved up to \(T=8\).  We test the \((s-1,s)\) first-subdiagonal Pad\'e approximants with \(s=3\) and \(s=4\), using
\begin{equation}
    \tau\in\{1.6,\ 0.8,\ 0.4,\ 0.2\}.
\end{equation}
The exact reference solution is \(\exp(TL)u(0)\).  Since the \((s-1,s)\) Pad\'e approximant has order \(2s-1\), the global solution error is expected to converge with order \(2s-1\) in this test.  Following~\cite{SunWeiWu2022EnergyLaws}, we define the total energy-dissipation accuracy as
\begin{equation}
\begin{aligned}
    \Delta E
    &:={}
    \left|
    \bigl(\|u(0)\|^2-\|u(T)\|^2\bigr)
    -
    \bigl(\|u^0\|^2-\|u^N\|^2\bigr)
    \right|  \\
    &=
    \left|\|u(T)\|^2-\|u^N\|^2\right|,
    \qquad N=T/\tau .
\end{aligned}
    \label{eq:total-energy-dissipation-error}
\end{equation}
For the energy-dissipation curve in Figure~\ref{fig:subdiag-energy-dissipation}(a), we use the \((2,3)\) Pad\'e approximant and \(\tau=0.2\).
\end{example}

\begin{table}[htbp]
\caption{The \(\ell^2\)-errors and the energy-dissipation accuracy \(\Delta E\) at \(T=8\), together with the corresponding convergence rates.}
\label{tab:subdiag-3by3-errors}
\centering
\small
\setlength{\tabcolsep}{3.5pt}
\renewcommand{\arraystretch}{1.15}
\begin{tabular}{c|cccc|cccc}
\hline
 & \multicolumn{4}{c|}{\((2,3)\) Pad\'e, \(s=3\)}
 & \multicolumn{4}{c}{\((3,4)\) Pad\'e, \(s=4\)}\\
\(\tau\) & \(\ell^2\) error & order & \(\Delta E\) & order
& \(\ell^2\) error & order & \(\Delta E\) & order\\
\hline
1.6 & \(1.55\times10^{-5}\) & --   & \(5.03\times10^{-7}\)  & --   & \(3.51\times10^{-7}\)  & --   & \(1.36\times10^{-8}\)  & --\\
0.8 & \(5.22\times10^{-7}\) & 4.89 & \(1.70\times10^{-8}\)  & 4.89 & \(2.91\times10^{-9}\)  & 6.91 & \(1.13\times10^{-10}\) & 6.91\\
0.4 & \(1.74\times10^{-8}\) & 4.90 & \(5.72\times10^{-10}\) & 4.89 & \(2.36\times10^{-11}\) & 6.95 & \(9.13\times10^{-13}\) & 6.95\\
0.2 & \(5.68\times10^{-10}\)& 4.94 & \(1.88\times10^{-11}\) & 4.93 & \(1.87\times10^{-13}\) & 6.97 & \(7.27\times10^{-15}\) & 6.97\\
\hline
\end{tabular}
\end{table}

\begin{example}
\label{ex:numerical-p1-convection}
The second test uses the seminegative ODE system arising from the piecewise linear discontinuous Galerkin semidiscretization of
\begin{equation}
    \psi_t+\psi_x=0
\end{equation}
on \([0,1]\) with periodic boundary conditions.  Let \(N_d=20\) and \(\Delta x=1/N_d\).  The semidiscrete system has the form
\begin{equation}
    \frac{\dd}{\dd t}u=Lu,
    \qquad
    u(t)\in\mathbb R^{2N_d},
    \qquad
    L=\frac1{\Delta x}
    \begin{pmatrix}
        L_1 & \sqrt3 L_1\\
        \sqrt3(2I_{N_d}-L_2) & -3L_2
    \end{pmatrix},
    \label{eq:p1-convection-system}
\end{equation}
where
\begin{equation}
L_1=\begin{pmatrix}
-1&&&1\\
1&-1&&\\
&\ddots&\ddots&\\
&&1&-1
\end{pmatrix},\qquad
L_2=\begin{pmatrix}
1&&&1\\
1&1&&\\
&\ddots&\ddots&\\
&&1&1
\end{pmatrix}.
\label{eq:L1-L2-def}
\end{equation}
The initial coefficient vector is the cellwise \(L^2\)-projection of \(\psi(x,0)=\sin(2\pi x)\) onto the modal basis \(\{1,\sqrt{3}\xi\}\), where \(\xi\in[-1,1]\) is the reference-cell coordinate.  We solve up to \(T=4\) by the \((1,2)\) first-subdiagonal Pad\'e approximant, with \(\tau=0.1\). This relatively large time step is used to illustrate the unconditional contractivity established above. The quantities \(\mathcal D_{\rm num}^n\) and \(\mathcal D_{\rm law}^n\) are plotted in Figure~\ref{fig:subdiag-energy-dissipation}(b) to verify the discrete energy identity \eqref{eq:main-energy-law}.
\end{example}

\begin{example}
\label{ex:numerical-p0-dispersion}
The third test follows the local discontinuous Galerkin semidiscretization of the dispersive equation
\begin{equation}
    \psi_t+\psi_{xxx}=0
\end{equation}
on \([0,1]\) with periodic boundary conditions.  With the same \(N_d=20\), \(\Delta x=1/N_d\), and the matrix \(L_1\) in \eqref{eq:L1-L2-def}, the semidiscrete system is
\begin{equation}
    \frac{\dd}{\dd t}u=Lu,
    \qquad
    u(t)\in\mathbb R^{N_d},
    \qquad
    L=\frac1{\Delta x^3}L_1L_1^T L_1^T.
    \label{eq:p0-dispersion-system}
\end{equation}
The initial coefficient vector is the cellwise \(L^2\)-projection of \(\psi(x,0)=\cos(2\pi x)\) onto piecewise constants, namely, the vector of cell averages.  We solve up to \(T=4\) using the \((3,4)\) first-subdiagonal Pad\'e approximant with \(\tau=0.1\). For this stiff semidiscrete operator, the computed dissipation quickly reaches the round-off level. Once this floating-point plateau is reached, direct subtraction of nearly equal energies is dominated by round-off error. In Figure~\ref{fig:subdiag-energy-dissipation}(c), we display only those time steps for which both dissipation values are positive and the relative difference between \(\mathcal D_{\rm num}^n\) and \(\mathcal D_{\rm law}^n\) is at most five percent.
\end{example}

\begin{figure}[H]
\centering
\includegraphics[width=.95\linewidth]{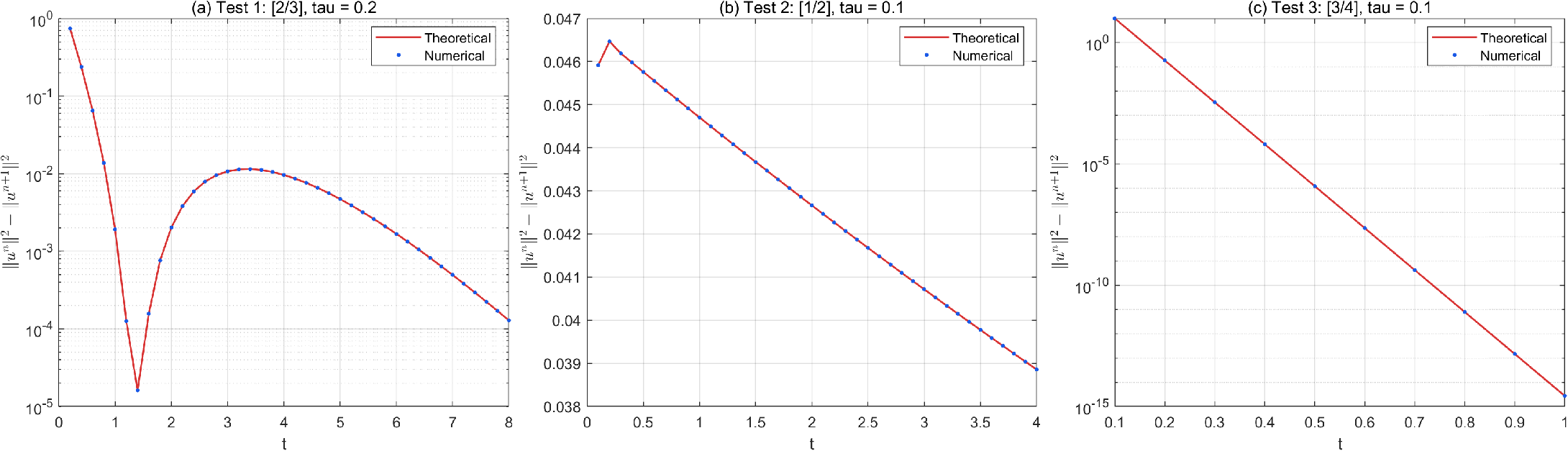}
\caption{Numerical energy dissipation magnitudes and the theoretical values given by the first-subdiagonal energy identity \eqref{eq:main-energy-law}. Panels (a), (b), and (c) correspond to Examples~\ref{ex:numerical-3by3}, \ref{ex:numerical-p1-convection}, and \ref{ex:numerical-p0-dispersion}, respectively.}
\label{fig:subdiag-energy-dissipation}
\end{figure}

\section{Conclusions}
\label{sec:conclusions}
We have proved an explicit discrete energy law for the first-subdiagonal Pad\'e approximants applied to linear seminegative problems with bounded operators.  The result extends the diagonal Pad\'e energy-law theory to the first-subdiagonal family and relies on an explicit Cholesky-type factorization of the corresponding energy coefficient matrix.  The main algebraic difficulty is that the matrix entries are alternating sums of Pad\'e coefficients, whereas the proposed triangular factor has parity-dependent factorial entries; this difficulty is resolved by reducing the matrix identity to scalar rational identities and proving them through product reductions and adjacent-index summations.  Together with the \(\beta\)-coefficient cancellation and the Pad\'e denominator root-location result, the factorization yields an exact energy law that recovers the classical unconditional contractivity for the class of problems considered here.  The numerical experiments illustrate the predicted order and verify the discrete dissipation identity.  For future work, it would be interesting to investigate whether a comparable explicit energy-law structure can be established for the second-subdiagonal Pad\'e family.

\clearpage
\appendix
\section{Algebraic verifications for Section~\ref{sec:first-subdiagonal-proof}}
\label{app:algebraic-verifications}

\subsection{Verification of Lemma~\ref{lem:nu-mu-connection}}
\label{app:nu-mu-verification}
This subsection gives the algebraic verification of the four parity identities used in Lemma~\ref{lem:nu-mu-connection}.  Factorials with noninteger arguments are interpreted through the Gamma function away from its poles, and the resulting rational identities are extended in \(s\) by the identity theorem for rational functions.  If the row index exceeds the column index, both sides vanish: the matrix \(U\) is upper triangular, and the factors \((-j)_i\) or \((1-j)_i\) in \eqref{eq:nu-oe}--\eqref{eq:nu-oo} vanish.  Hence we only consider the upper-triangular range.

We write out the two off-diagonal parity cases, which contain the only nontrivial factorial reductions.  The two same-parity cases are then obtained from the same-parity formula \eqref{eq:mu-even} by the same simplification.

\paragraph{The odd--even entry.}
Substituting the odd-parity formula \eqref{eq:mu-odd} with row index \(2i-2\) and column index \(2j-1\) gives
\begin{equation}
\sqrt{d_{2i-2}}\,\mu_{2i-2,2j-1}
=-\frac{2s!\sqrt{4i-3}}{(2s)!}
\frac{(2s+2i-2j-3)!(s-i-j)!(i+j-2)!}
{(s-2j)!(s-j+i-2)!(2i+2j-3)!(j-i)!}.
\label{eq:app-oe-start-short}
\end{equation}
The two factorial blocks are reduced by Lemma~\ref{lem:pochhammer-basic} as
\begin{align}
\frac{(2s+2i-2j-3)!(s-i-j)!}{(s-j+i-2)!}
&=-(2s-2j-1)!\,2^{2i-2}
\frac{\left(s-j+\frac12\right)_{i-1}}{(j-s+1)_{i-1}},
\label{eq:app-oe-block1}\\*
\frac{(i+j-2)!}{(2i+2j-3)!(j-i)!}
&=
\frac{(-j)_i}{(2j)!\,2^{2i-2}\left(j+\frac12\right)_{i-1}}.
\label{eq:app-oe-block2}
\end{align}
Using
\begin{equation}
\eta_{2j}^{(s)}=\frac{s!}{(2s-1)!}
\frac{(2s-2j-1)!}{(2j)!(s-2j)!},
\end{equation}
we obtain
\begin{equation}
\sqrt{d_{2i-2}}\,\mu_{2i-2,2j-1}
=
\frac{2\sqrt{4i-3}}{s}
\frac{\left(s-j+\frac12\right)_{i-1}(-j)_i}
{(j-s+1)_{i-1}\left(j+\frac12\right)_{i-1}}
\eta_{2j}^{(s)},
\end{equation}
which is \eqref{eq:nu-oe}.

\paragraph{The even--odd entry.}
Substituting \eqref{eq:mu-odd} with row index \(2i-1\) and column index \(2j-2\) gives
\begin{equation}
\sqrt{d_{2i-1}}\,\mu_{2i-1,2j-2}
=-\frac{2s!\sqrt{4i-1}}{(2s)!}
\frac{(2s+2i-2j-1)!(s-i-j)!(i+j-2)!}
{(s-2j+1)!(s-j+i-1)!(2i+2j-3)!(j-i-1)!}.
\label{eq:app-eo-start-short}
\end{equation}
The corresponding block reductions are
\begin{align}
\frac{(2s+2i-2j-1)!(s-i-j)!}{(s-j+i-1)!}
&=(2s-2j)!\,2^{2i-1}
\frac{\left(s-j+\frac12\right)_i}{(j-s)_i},
\label{eq:app-eo-block1}\\
\frac{(i+j-2)!}{(2i+2j-3)!(j-i-1)!}
&=
\frac{(1-j)_i}{(2j-1)!\,2^{2i-1}\left(j+\frac12\right)_{i-1}}.
\label{eq:app-eo-block2}
\end{align}
Together with
\begin{equation}
\eta_{2j-1}^{(s)}=\frac{s!}{(2s-1)!}
\frac{(2s-2j)!}{(2j-1)!(s-2j+1)!},
\end{equation}
these reductions give
\begin{equation}
\sqrt{d_{2i-1}}\,\mu_{2i-1,2j-2}
=
-\frac{2\sqrt{4i-1}}{s}
\frac{\left(s-j+\frac12\right)_i(1-j)_i}
{(j-s)_i\left(j+\frac12\right)_{i-1}}
\eta_{2j-1}^{(s)},
\end{equation}
which is \eqref{eq:nu-eo}.

\paragraph{The same-parity entries.}
Applying the same-parity formula \eqref{eq:mu-even} and using the same two block reductions gives, for the even--even entry,
\begin{equation}
\sqrt{d_{2i-1}}\,\mu_{2i-1,2j-1}
=2\sqrt{4i-1}
\frac{\left(s+\frac12-j\right)_i(-j)_i}
{(j-s)_i\left(j+\frac12\right)_i}
\frac{s-j}{s}\eta_{2j}^{(s)}.
\end{equation}
Since \((j-s)_i=-(s-j)(j-s+1)_{i-1}\), this is exactly \eqref{eq:nu-ee}.  Similarly, the odd--odd entry gives
\begin{equation}
\sqrt{d_{2i-2}}\,\mu_{2i-2,2j-2}
=
\frac{2\sqrt{4i-3}}{s}
\frac{\left(s+\frac12-j\right)_i(1-j)_{i-1}}
{(j-s)_{i-1}\left(j+\frac12\right)_{i-1}}
\eta_{2j-1}^{(s)},
\end{equation}
which is \eqref{eq:nu-oo}.  The four parity cases prove Lemma~\ref{lem:nu-mu-connection}.

\subsection{Polynomial checks for Lemma~\ref{lem:telescoping-certificates}}
\label{app:telescoping-verification}
For each type, let \(\rho_n^{(r)}\) be the explicit rational factor displayed below. The Pochhammer recurrence \((x)_{n+1}=(x)_n(x+n)\) gives
\begin{equation*}
    H_{n+1}^{(r)}=\rho_n^{(r)}H_n^{(r)}.
\end{equation*}
Consequently, the polynomial identity to be checked is
\begin{equation}
    \cC_{3,n+1}^{(r)}\rho_n^{(r)}-\cC_{3,n}^{(r)}
    +\cC_{1,n}^{(r)}+\cC_{2,n}^{(r)}=0.
    \label{eq:appendix-polynomial-identity}
\end{equation}
The four rational factors are
\begin{align*}
\rho_n^{(1)}&=
\frac{
\left(s+\frac32-p+n\right)(1-p+n)
\left(s+\frac12-q+n\right)(-q+n)
}{
(p-s+1+n)\left(p+\frac32+n\right)
(q-s+1+n)\left(q+\frac32+n\right)},\\
\rho_n^{(2)}&=
\frac{
\left(s-p+\frac12+n\right)(1-p+n)
\left(s-q+\frac12+n\right)(1-q+n)
}{
(p-s+1+n)\left(p+\frac12+n\right)
(q-s+2+n)\left(q+\frac32+n\right)},\\
\rho_n^{(3)}&=
\frac{
\left(s+\frac12-p+n\right)(-p+n)
\left(s+\frac12-q+n\right)(1-q+n)
}{
(p-s+1+n)\left(p+\frac32+n\right)
(q-s+2+n)\left(q+\frac32+n\right)},\\
\rho_n^{(4)}&=
\frac{
\left(s-p+\frac12+n\right)(1-p+n)
\left(s+\frac32-q+n\right)(1-q+n)
}{
(p-s+1+n)\left(p+\frac12+n\right)
(q-s+1+n)\left(q+\frac32+n\right)}.
\end{align*}
Substitution of each displayed factor and the corresponding definitions of \(\cC_{1,n}^{(r)},\cC_{2,n}^{(r)},\cC_{3,n}^{(r)}\) into \eqref{eq:appendix-polynomial-identity} reduces, after clearing the displayed denominators, to the zero polynomial in \(n,p,q,s\). This is a finite algebraic verification; no limiting or summation argument is involved in this step.

\subsection{Remaining product reductions}
\label{app:product-reductions}
For \eqref{eq:prod2}, direct substitution gives
\begin{align*}
&\nu_{2i,2p-1}^{(s)}\nu_{2i,2q+1}^{(s)}
=
A_2 H_{i-1}^{(2)}
\frac{\cC_{1,i-1}^{(2)}}{\Delta_2},\\
&\nu_{2i-1,2p}^{(s)}\nu_{2i-1,2q}^{(s)}
=
A_2 H_{i-1}^{(2)}
\frac{\cC_{2,i-1}^{(2)}}{\Delta_2},
\end{align*}
where
\[
A_2=
\frac{(2p-1)q}{s^2}
\eta_{2p-1}^{(s)}\eta_{2q}^{(s)},
\qquad
\Delta_2=(2q+1)(2p-1)(p-s)(q-s+1).
\]
Adding the two lines and using \eqref{eq:PsiPhi2} gives \eqref{eq:prod2}.

For \eqref{eq:prod3}, direct substitution gives
\begin{align*}
&\nu_{2i,2q+1}^{(s)}\nu_{2i,2p}^{(s)}
=
A_3 H_{i-1}^{(3)}
\frac{\cC_{1,i-1}^{(3)}}{\Delta_3},\\
&\nu_{2i-1,2p+1}^{(s)}\nu_{2i-1,2q}^{(s)}
=
A_3 H_{i-1}^{(3)}
\frac{\cC_{2,i-1}^{(3)}}{\Delta_3},
\end{align*}
where
\[
A_3=
\frac{2q(p-s)}{s^2}
\eta_{2p}^{(s)}\eta_{2q}^{(s)},
\qquad
\Delta_3=(2p+1)(p-s)(q-s+1)(2q+1).
\]
Adding the two lines and using \eqref{eq:PsiPhi3} gives \eqref{eq:prod3}.

For \eqref{eq:prod4}, direct substitution gives
\begin{align*}
&\nu_{2i-1,2p}^{(s)}\nu_{2i-1,2q-1}^{(s)}
=
A_4 H_{i-1}^{(4)}
\frac{\cC_{2,i-1}^{(4)}}{\Delta_4},\\
&\nu_{2i,2p-1}^{(s)}\nu_{2i,2q}^{(s)}
=
A_4 H_{i-1}^{(4)}
\frac{\cC_{1,i-1}^{(4)}}{\Delta_4},
\end{align*}
where
\[
A_4=
\frac{(1-2p)(2s-2q+1)}{2s^2}
\eta_{2p-1}^{(s)}\eta_{2q-1}^{(s)},
\qquad
\Delta_4=(2p-1)(p-s)(q-s)(2q+1).
\]
Adding the two lines and using \eqref{eq:PsiPhi4} gives \eqref{eq:prod4}. This completes the verification of the remaining product reductions.

\section*{Acknowledgments}
The work of Miaosen Jiao and Kailiang Wu was partially supported by Science Challenge Project (No.~TZ2025007) and the Shenzhen Science and Technology Program (Grant Nos.~JCYJ20250604144300001 and RCJC20221008092757098).

\section*{Conflicts of Interest}
The authors declare no conflict of interest.

\end{document}